\documentclass{article}

\usepackage{tikz}
\usepackage{array}
\usepackage{amsmath}
\usepackage{amsthm}
\usepackage{amssymb}
\usepackage{multirow}
\usepackage{cite}
\usepackage{float} \usetikzlibrary{decorations.pathreplacing}
\newcommand{\Asterisk}{\mathop{\scalebox{1.5}{\raisebox{0.0ex}{{\footnotesize{$\ast$}}}}}}%
\newcommand\unv[1]{
\phantom{\Asterisk} \hfill #1 \hfill \Asterisk
}

\tikzstyle{smallgraph}=[
  every node/.style={circle,draw,fill=black!30,inner sep=0pt,minimum width=4pt},
  every path/.append style={semithick}
]
\tikzstyle{largegraph}=[
  every node/.style={circle,draw,fill=black,inner sep=0pt,minimum width=8pt},
  every path/.append style={semithick}
]

\tikzstyle{every node}=[circle, draw, fill=black!100,
                        inner sep=0pt, minimum width=6pt]
\tikzstyle{altfill}=[fill=white]

\pgfmathsetmacro{\threevradius}{0.375*sqrt(2/3)}
\pgfmathsetmacro{\fourvradius}{0.375}
\pgfmathsetmacro{\fivevradius}{0.375}
\pgfmathsetmacro{\sixvspacing}{0.36}
\pgfmathsetmacro{\sevenvradius}{0.36}
\pgfmathsetmacro{\eightvradius}{0.5}
\pgfmathsetmacro{\topgspacing}{0.15}
\pgfmathsetmacro{\btmgspacing}{0.15}

\newcommand\fourvertex{%
\foreach \n/\a in {1/135,2/45,3/225,4/315}{\node (\n) at (\a:\fourvradius) {};}%
\path (1) -- ++(90:\topgspacing) node[draw=none,fill=none,minimum width=0pt] () {};%
\path (3) -- ++(270:\btmgspacing) node[draw=none,fill=none,minimum width=0pt] () {};%
}
\newcommand\fivevertex{%
\foreach \n/\a in {1/135,2/45,3/225,4/315}{\node (\n) at (\a:\fivevradius) {};} \node (5) at (0,0) {};%
\path (2) -- ++(90:\topgspacing) node[draw=none,fill=none,minimum width=0pt] () {};%
\path (3) -- ++(270:\btmgspacing) node[draw=none,fill=none,minimum width=0pt] () {};%
}
\newcommand\sixvertex{%
\node (1) at (0,0) {}; \def\lastn{1}\foreach \n/\a [remember=\n as \lastn] in%
    {2/0,3/315,4/225,5/180,6/135}{\path (\lastn) -- ++(\a:\sixvspacing) node (\n) {};}%
\path (1) -- ++(90:\topgspacing) node[draw=none,fill=none,minimum width=0pt] () {};%
\path (5) -- ++(270:\btmgspacing) node[draw=none,fill=none,minimum width=0pt] () {};%
}

\makeatletter
\newcommand\addcase[3]{\expandafter\def\csname\string#1@case@#2\endcsname{#3}}
\newcommand\makeswitch[2][]{%
  \newcommand#2[1]{%
    \ifcsname\string#2@case@##1\endcsname\csname\string#2@case@##1\endcsname\else#1\fi%
  }%
}
\makeatother

\newcommand{\ctikz}[1]{\ensuremath{\vcenter{\hbox{#1}}}}

\makeswitch[Missing graph]\ngraph

\addcase\ngraph{4.1}{\begin{tikzpicture}[smallgraph]\fourvertex\draw (1) -- (3) -- (4); \draw (3) -- (2);\end{tikzpicture} }
\addcase\ngraph{4.2}{\begin{tikzpicture}[smallgraph]\fourvertex\draw (1) -- (3) -- (4) -- (2);\end{tikzpicture}}
\addcase\ngraph{4.3}{\begin{tikzpicture}[smallgraph]\fourvertex\draw (1) -- (3) -- (4) -- (1); \draw (4) -- (2);\end{tikzpicture}}
\addcase\ngraph{4.4}{\begin{tikzpicture}[smallgraph]\fourvertex\draw (1) -- (3) -- (4) -- (2) -- (1);\end{tikzpicture}}
\addcase\ngraph{4.5}{\begin{tikzpicture}[smallgraph]\fourvertex\draw (1) -- (3) -- (4) -- (1); \draw (4) -- (2) -- (1);\end{tikzpicture}}
\addcase\ngraph{4.6}{\begin{tikzpicture}[smallgraph]\fourvertex\draw (1) -- (3) -- (4) -- (1); \draw (4) -- (2) -- (1); \draw (2) -- (3);\end{tikzpicture}}

\addcase\ngraph{5.1}{\begin{tikzpicture}[smallgraph]\fivevertex\draw (3) -- (1) -- (5) -- (4) -- (2);\end{tikzpicture}}
\addcase\ngraph{5.2}{\begin{tikzpicture}[smallgraph]\fivevertex\foreach \n in {1,2,3,4} {\draw (\n) -- (5);}\end{tikzpicture}}
\addcase\ngraph{5.3}{\begin{tikzpicture}[smallgraph]\fivevertex\draw (1) -- (2) -- (5) -- (3); \draw (5) -- (4);\end{tikzpicture}}
\addcase\ngraph{5.4}{\begin{tikzpicture}[smallgraph]\fivevertex\draw (1) -- (2) -- (5) -- (3) -- (4) -- (5);\end{tikzpicture}}
\addcase\ngraph{5.5}{\begin{tikzpicture}[smallgraph]\fivevertex\draw (3) -- (1) -- (5) -- (2) -- (4); \draw (1) -- (2);\end{tikzpicture}}
\addcase\ngraph{5.6}{\begin{tikzpicture}[smallgraph]\fivevertex\draw (1) -- (2); \foreach \n in {1,2,3,4} {\draw (\n) -- (5);}\end{tikzpicture}}
\addcase\ngraph{5.7}{\begin{tikzpicture}[smallgraph]\fivevertex\draw (1) -- (2) -- (4) -- (3) -- (1) -- (5);\end{tikzpicture}}
\addcase\ngraph{5.8}{\begin{tikzpicture}[smallgraph]\fivevertex\draw (1) -- (5) -- (2) -- (4) -- (3) -- (1);\end{tikzpicture}}
\addcase\ngraph{5.9}{\begin{tikzpicture}[smallgraph]\fivevertex\draw (1) -- (2) -- (5) -- (4) -- (3) -- (5) -- (1);\end{tikzpicture}}
\addcase\ngraph{5.10}{\begin{tikzpicture}[smallgraph]\fivevertex\draw (1) -- (2) -- (4) -- (3) -- (1) -- (5) -- (4);\end{tikzpicture}}
\addcase\ngraph{5.11}{\begin{tikzpicture}[smallgraph]\fivevertex\draw (1) -- (2) -- (5) -- (3) -- (1) -- (5) -- (4);\end{tikzpicture}}
\addcase\ngraph{5.12}{\begin{tikzpicture}[smallgraph]\fivevertex\draw (1) -- (2) -- (5) -- (3) -- (1) -- (5); \draw (2) -- (4);\end{tikzpicture}}
\addcase\ngraph{5.13}{\begin{tikzpicture}[smallgraph]\fivevertex\draw (1) -- (5) -- (2) -- (4) -- (3) -- (1) -- (2);\end{tikzpicture}}
\addcase\ngraph{5.14}{\begin{tikzpicture}[smallgraph]\fivevertex\draw (2) -- (1) -- (3) -- (4); \foreach \n in {1,2,3,4} {\draw (\n) -- (5);}\end{tikzpicture}}
\addcase\ngraph{5.15}{\begin{tikzpicture}[smallgraph]\fivevertex\draw (1) -- (2) -- (4) -- (3) -- (1) -- (5) -- (4); \draw[bend left] (1) to (4);\end{tikzpicture}}
\addcase\ngraph{5.16}{\begin{tikzpicture}[smallgraph]\fivevertex\draw (2) -- (1) -- (3) -- (4) -- (5) -- (1); \draw (3) -- (5); \draw[bend left] (1) to (4);\end{tikzpicture}}
\addcase\ngraph{5.17}{\begin{tikzpicture}[smallgraph]\fivevertex\draw (1) -- (2) -- (4) -- (3) -- (1) -- (5) -- (3); \draw (5) -- (4);\end{tikzpicture}}
\addcase\ngraph{5.18}{\begin{tikzpicture}[smallgraph]\fivevertex\draw (1) -- (2) -- (4) -- (3) -- (1) -- (5) -- (3); \draw (5) -- (4); \draw[bend left] (1) to (4);\end{tikzpicture}}
\addcase\ngraph{5.19}{\begin{tikzpicture}[smallgraph]\fivevertex\draw (1) -- (2) -- (4) -- (3) -- (1); \foreach \n in {1,2,3,4} {\draw (\n) -- (5);}\end{tikzpicture}}
\addcase\ngraph{5.20}{\begin{tikzpicture}[smallgraph]\fivevertex\draw (1) -- (2) -- (4) -- (3) -- (1); \foreach \n in {1,2,3,4} {\draw (\n) -- (5);} \draw[bend left] (1) to (4);\end{tikzpicture}}
\addcase\ngraph{5.21}{\begin{tikzpicture}[smallgraph]\fivevertex\draw (1) -- (2) -- (4) -- (3) -- (1); \foreach \n in {1,2,3,4} {\draw (\n) -- (5);} \draw[bend left] (1) to (4); \draw[bend left] (3) to (2);\end{tikzpicture}}

\addcase\ngraph{6.1}{\begin{tikzpicture}[smallgraph]\sixvertex\end{tikzpicture}}
\addcase\ngraph{6.2}{\begin{tikzpicture}[smallgraph]\sixvertex \draw (1) -- (2);\end{tikzpicture}}
\addcase\ngraph{6.3}{\begin{tikzpicture}[smallgraph]\sixvertex \draw (6) -- (1) -- (2);\end{tikzpicture}}
\addcase\ngraph{6.4}{\begin{tikzpicture}[smallgraph]\sixvertex \draw (1) -- (2); \draw (5) -- (6);\end{tikzpicture}}
\addcase\ngraph{6.5}{\begin{tikzpicture}[smallgraph]\sixvertex \draw (6) -- (1) -- (2) -- (3);\end{tikzpicture}}
\addcase\ngraph{6.6}{\begin{tikzpicture}[smallgraph]\sixvertex \draw (1) -- (5) -- (6) -- (1);\end{tikzpicture}}
\addcase\ngraph{6.7}{\begin{tikzpicture}[smallgraph]\sixvertex \draw (6) -- (1) -- (2); \draw (1) -- (5);\end{tikzpicture}}
\addcase\ngraph{6.8}{\begin{tikzpicture}[smallgraph]\sixvertex \draw (1) -- (6) -- (5); \draw (2) -- (3);\end{tikzpicture}}
\addcase\ngraph{6.9}{\begin{tikzpicture}[smallgraph]\sixvertex \draw (1) -- (6); \draw (2) -- (3); \draw (4) -- (5);\end{tikzpicture}}
\addcase\ngraph{6.10}{\begin{tikzpicture}[smallgraph]\sixvertex \draw (2) -- (1) -- (6) -- (5) -- (1);\end{tikzpicture}}
\addcase\ngraph{6.11}{\begin{tikzpicture}[smallgraph]\sixvertex \draw (1) -- (5) -- (4) -- (2) -- (1);\end{tikzpicture}}
\addcase\ngraph{6.12}{\begin{tikzpicture}[smallgraph]\sixvertex \draw (5) -- (6) -- (1) -- (2) -- (3);\end{tikzpicture}}
\addcase\ngraph{6.13}{\begin{tikzpicture}[smallgraph]\sixvertex \draw (6) -- (1) -- (5); \draw (2) -- (1) -- (4);\end{tikzpicture}}
\addcase\ngraph{6.14}{\begin{tikzpicture}[smallgraph]\sixvertex \draw (6) -- (1) -- (2) -- (3); \draw (1) -- (5);\end{tikzpicture}}
\addcase\ngraph{6.15}{\begin{tikzpicture}[smallgraph]\sixvertex \draw (1) -- (5) -- (6) -- (1); \draw (2) -- (3);\end{tikzpicture}}
\addcase\ngraph{6.16}{\begin{tikzpicture}[smallgraph]\sixvertex \draw (6) -- (1) -- (2); \draw (3) -- (4) -- (5);\end{tikzpicture}}
\addcase\ngraph{6.17}{\begin{tikzpicture}[smallgraph]\sixvertex \draw (6) -- (1) -- (2) -- (3); \draw (4) -- (5);\end{tikzpicture}}
\addcase\ngraph{6.18}{\begin{tikzpicture}[smallgraph]\sixvertex \draw (6) -- (1) -- (2); \draw (3) -- (4); \draw (1) -- (5);\end{tikzpicture}}
\addcase\ngraph{6.19}{\begin{tikzpicture}[smallgraph]\sixvertex \draw (2) -- (1) -- (6) -- (5) -- (4); \draw (1) -- (5);\end{tikzpicture}}
\addcase\ngraph{6.20}{\begin{tikzpicture}[smallgraph]\sixvertex \draw (1) -- (2) -- (4) -- (5) -- (6) -- (1);\end{tikzpicture}}
\addcase\ngraph{6.21}{\begin{tikzpicture}[smallgraph]\sixvertex \draw (3) -- (2) -- (1) -- (6) -- (5) -- (1);\end{tikzpicture}}
\addcase\ngraph{6.22}{\begin{tikzpicture}[smallgraph]\sixvertex \draw (6) -- (1) -- (2) -- (4) -- (5) -- (1);\end{tikzpicture}}
\addcase\ngraph{6.23}{\begin{tikzpicture}[smallgraph]\sixvertex \draw (1) -- (2) -- (4) -- (5) -- (1) -- (4);\end{tikzpicture}}
\addcase\ngraph{6.24}{\begin{tikzpicture}[smallgraph]\sixvertex \draw (2) -- (1) -- (6) -- (5) -- (1) -- (4);\end{tikzpicture}}
\addcase\ngraph{6.25}{\begin{tikzpicture}[smallgraph]\sixvertex \draw (2) -- (1) -- (6) -- (5) -- (4) -- (3);\end{tikzpicture}}
\addcase\ngraph{6.26}{\begin{tikzpicture}[smallgraph]\sixvertex \draw (1) -- (6) -- (5) -- (4) -- (3); \draw (2) -- (4);\end{tikzpicture}}
\addcase\ngraph{6.27}{\begin{tikzpicture}[smallgraph]\sixvertex \draw (1) -- (6) -- (5) -- (4); \draw (2) -- (5) -- (3);\end{tikzpicture}}
\addcase\ngraph{6.28}{\begin{tikzpicture}[smallgraph]\sixvertex \draw (3) -- (2) -- (1) -- (5) -- (4); \draw (1) -- (6);\end{tikzpicture}}
\addcase\ngraph{6.29}{\begin{tikzpicture}[smallgraph]\sixvertex \draw (6) -- (1) -- (5); \draw (1) -- (2) -- (3); \draw (2) -- (4);\end{tikzpicture}}
\addcase\ngraph{6.30}{\begin{tikzpicture}[smallgraph]\sixvertex \draw (1) -- (5) -- (6) -- (1); \draw (2) -- (3) -- (4);\end{tikzpicture}}
\addcase\ngraph{6.31}{\begin{tikzpicture}[smallgraph]\sixvertex \draw (1) -- (2) -- (3); \draw (4) -- (2) -- (5); \draw (2) -- (6);\end{tikzpicture}}
\addcase\ngraph{6.32}{\begin{tikzpicture}[smallgraph]\sixvertex \draw (1) -- (4) -- (5) -- (6) -- (1); \draw (2) -- (3);\end{tikzpicture}}
\addcase\ngraph{6.33}{\begin{tikzpicture}[smallgraph]\sixvertex \draw (4) -- (5) -- (6) -- (1) -- (5); \draw (2) -- (3);\end{tikzpicture}}
\addcase\ngraph{6.34}{\begin{tikzpicture}[smallgraph]\sixvertex \draw (1) -- (2) -- (4) -- (5) -- (1) -- (4); \draw (2) -- (5);\end{tikzpicture}}
\addcase\ngraph{6.35}{\begin{tikzpicture}[smallgraph]\sixvertex \draw (1) -- (2) -- (4) -- (5) -- (6) -- (1) -- (5);\end{tikzpicture}}
\addcase\ngraph{6.36}{\begin{tikzpicture}[smallgraph]\sixvertex \draw (1) -- (2) -- (4) -- (1) -- (5) -- (6) -- (1);\end{tikzpicture}}
\addcase\ngraph{6.37}{\begin{tikzpicture}[smallgraph]\sixvertex \draw (3) -- (2) -- (1) -- (5) -- (4) -- (2) -- (5);\end{tikzpicture}}
\addcase\ngraph{6.38}{\begin{tikzpicture}[smallgraph]\sixvertex \draw (3) -- (4) -- (5) -- (1) -- (2) -- (4); \draw (2) -- (5);\end{tikzpicture}}
\addcase\ngraph{6.39}{\begin{tikzpicture}[smallgraph]\sixvertex \draw (1) -- (2) -- (3) -- (5) -- (1); \draw (2) -- (4) -- (5);\end{tikzpicture}}
\addcase\ngraph{6.40}{\begin{tikzpicture}[smallgraph]\sixvertex \draw (1) -- (2) -- (3) -- (4) -- (5) -- (6) -- (1);\end{tikzpicture}}
\addcase\ngraph{6.41}{\begin{tikzpicture}[smallgraph]\sixvertex \draw (2) -- (3) -- (4) -- (5) -- (6) -- (1) -- (3);\end{tikzpicture}}
\addcase\ngraph{6.42}{\begin{tikzpicture}[smallgraph]\sixvertex \draw (2) -- (3) -- (4) -- (5) -- (6) -- (1) -- (4);\end{tikzpicture}}
\addcase\ngraph{6.43}{\begin{tikzpicture}[smallgraph]\sixvertex \draw (2) -- (3) -- (4) -- (5) -- (6) -- (1) -- (5);\end{tikzpicture}}
\addcase\ngraph{6.44}{\begin{tikzpicture}[smallgraph]\sixvertex \draw (2) -- (1) -- (6) -- (5) -- (4) -- (3); \draw (1) -- (4);\end{tikzpicture}}
\addcase\ngraph{6.45}{\begin{tikzpicture}[smallgraph]\sixvertex \draw (2) -- (1) -- (6) -- (5) -- (4) -- (1) -- (3);\end{tikzpicture}}
\addcase\ngraph{6.46}{\begin{tikzpicture}[smallgraph]\sixvertex \draw (3) -- (2) -- (1) -- (6) -- (5) -- (1); \draw (4) -- (5);\end{tikzpicture}}
\addcase\ngraph{6.47}{\begin{tikzpicture}[smallgraph]\sixvertex \draw (3) -- (2) -- (1) -- (6) -- (5) -- (1) -- (4);\end{tikzpicture}}
\addcase\ngraph{6.48}{\begin{tikzpicture}[smallgraph]\sixvertex \draw (2) -- (1) -- (6) -- (5) -- (1) -- (3); \draw (1) -- (4);\end{tikzpicture}}
\addcase\ngraph{6.49}{\begin{tikzpicture}[smallgraph]\sixvertex \draw (3) -- (2) -- (1) -- (5) -- (4) -- (2); \draw (5) -- (6);\end{tikzpicture}}
\addcase\ngraph{6.50}{\begin{tikzpicture}[smallgraph]\sixvertex \draw (1) -- (5) -- (6) -- (1); \draw (2) -- (3) -- (4) -- (2);\end{tikzpicture}}
\addcase\ngraph{6.51}{\begin{tikzpicture}[smallgraph]\sixvertex \draw (6) -- (1) -- (2) -- (5) -- (1); \draw (2) -- (3); \draw (4) -- (5);\end{tikzpicture}}
\addcase\ngraph{6.52}{\begin{tikzpicture}[smallgraph]\sixvertex \draw (1) -- (4) -- (5) -- (1) -- (6) -- (5); \draw (2) -- (3);\end{tikzpicture}}
\addcase\ngraph{6.53}{\begin{tikzpicture}[smallgraph]\sixvertex \draw (1) -- (5) -- (6) -- (1) -- (2) -- (3); \draw (2) -- (4);\end{tikzpicture}}
\addcase\ngraph{6.54}{\begin{tikzpicture}[smallgraph]\sixvertex \draw (2) -- (1) -- (3); \draw (4) -- (5) -- (6) -- (1) -- (5);\end{tikzpicture}}
\addcase\ngraph{6.55}{\begin{tikzpicture}[smallgraph]\sixvertex \draw (1) -- (2) -- (4) -- (5) -- (6) -- (1) -- (5); \draw (2) -- (6);\end{tikzpicture}}
\addcase\ngraph{6.56}{\begin{tikzpicture}[smallgraph]\sixvertex \draw (1) -- (2) -- (4) -- (1) -- (6) -- (5) -- (1); \draw (4) -- (5);\end{tikzpicture}}
\addcase\ngraph{6.57}{\begin{tikzpicture}[smallgraph]\sixvertex \draw (3) -- (4) -- (5) -- (1) -- (2) -- (4) -- (1); \draw (2) -- (5);\end{tikzpicture}}
\addcase\ngraph{6.58}{\begin{tikzpicture}[smallgraph]\sixvertex \draw (2) -- (1) -- (6) -- (5) -- (4) -- (1) -- (5) -- (2);\end{tikzpicture}}
\addcase\ngraph{6.59}{\begin{tikzpicture}[smallgraph]\sixvertex \draw (2) -- (3) -- (4) -- (5) -- (2) -- (1) -- (6) -- (5);\end{tikzpicture}}
\addcase\ngraph{6.60}{\begin{tikzpicture}[smallgraph]\sixvertex \draw (2) -- (3) -- (4) -- (2) -- (1) -- (6) -- (5) -- (4);\end{tikzpicture}}
\addcase\ngraph{6.61}{\begin{tikzpicture}[smallgraph]\sixvertex \draw (6) -- (1) -- (3) -- (2) -- (6) -- (5) -- (4) -- (3);\end{tikzpicture}}
\addcase\ngraph{6.62}{\begin{tikzpicture}[smallgraph]\sixvertex \draw (3) -- (2) -- (1) -- (6) -- (5) -- (4) -- (1) -- (5);\end{tikzpicture}}
\addcase\ngraph{6.63}{\begin{tikzpicture}[smallgraph]\sixvertex \draw (3) -- (2) -- (4) -- (5) -- (6) -- (1) -- (5); \draw (1) -- (4);\end{tikzpicture}}
\addcase\ngraph{6.64}{\begin{tikzpicture}[smallgraph]\sixvertex \draw (3) -- (4) -- (5) -- (6) -- (1) -- (5) -- (2) -- (4);\end{tikzpicture}}
\addcase\ngraph{6.65}{\begin{tikzpicture}[smallgraph]\sixvertex \draw (6) -- (1) -- (2) -- (3); \draw (1) -- (5) -- (2) -- (4) -- (5);\end{tikzpicture}}
\addcase\ngraph{6.66}{\begin{tikzpicture}[smallgraph]\sixvertex \draw (3) -- (4) -- (5) -- (6) -- (1) -- (2) -- (4); \draw (1) -- (5);\end{tikzpicture}}
\addcase\ngraph{6.67}{\begin{tikzpicture}[smallgraph]\sixvertex \draw (4) -- (3) -- (2) -- (4) -- (1) -- (6) -- (5) -- (4);\end{tikzpicture}}
\addcase\ngraph{6.68}{\begin{tikzpicture}[smallgraph]\sixvertex \draw (6) -- (5) -- (4) -- (3) -- (1) -- (2) -- (4); \draw (1) -- (5);\end{tikzpicture}}
\addcase\ngraph{6.69}{\begin{tikzpicture}[smallgraph]\sixvertex \draw (1) -- (4) -- (5) -- (6) -- (1) -- (5); \draw (2) -- (3); \draw (4) -- (6);\end{tikzpicture}}
\addcase\ngraph{6.70}{\begin{tikzpicture}[smallgraph]\sixvertex \draw (3) -- (2) -- (1) -- (6) -- (5) -- (4) -- (1); \draw (2) -- (4);\end{tikzpicture}}
\addcase\ngraph{6.71}{\begin{tikzpicture}[smallgraph]\sixvertex \draw (6) -- (1) -- (2) -- (4) -- (5) -- (1) -- (3) -- (4);\end{tikzpicture}}
\addcase\ngraph{6.72}{\begin{tikzpicture}[smallgraph]\sixvertex \draw (2) -- (4) -- (3); \draw (4) -- (1) -- (6) -- (4) -- (5) -- (6);\end{tikzpicture}}
\addcase\ngraph{6.73}{\begin{tikzpicture}[smallgraph]\sixvertex \draw (2) -- (5) -- (1) -- (6) -- (5) -- (3) -- (4) -- (5);\end{tikzpicture}}
\addcase\ngraph{6.74}{\begin{tikzpicture}[smallgraph]\sixvertex \draw (3) -- (4) -- (5) -- (6) -- (1) -- (2) -- (4) -- (1);\end{tikzpicture}}
\addcase\ngraph{6.75}{\begin{tikzpicture}[smallgraph]\sixvertex \draw (6) -- (1) -- (2) -- (4) -- (3); \draw (1) -- (5) -- (4); \draw (2) -- (5);\end{tikzpicture}}
\addcase\ngraph{6.76}{\begin{tikzpicture}[smallgraph]\sixvertex \draw (6) -- (5) -- (1) -- (2) -- (5) -- (4) -- (2) -- (3);\end{tikzpicture}}
\addcase\ngraph{6.77}{\begin{tikzpicture}[smallgraph]\sixvertex \draw (5) -- (1) -- (6) -- (5) -- (4) -- (3) -- (2) -- (4);\end{tikzpicture}}
\addcase\ngraph{6.78}{\begin{tikzpicture}[smallgraph]\sixvertex \draw (3) -- (2) -- (1) -- (6) -- (5) -- (1); \draw (5) -- (2) -- (4);\end{tikzpicture}}
\addcase\ngraph{6.79}{\begin{tikzpicture}[smallgraph]\sixvertex \draw (4) -- (1) -- (6) -- (4) -- (5) -- (6); \draw (2) -- (3) -- (4) -- (2);\end{tikzpicture}}
\addcase\ngraph{6.80}{\begin{tikzpicture}[smallgraph]\sixvertex \draw (3) -- (4) -- (5) -- (6) -- (1) -- (2) -- (4) -- (1); \draw (4) -- (6);\end{tikzpicture}}
\addcase\ngraph{6.81}{\begin{tikzpicture}[smallgraph]\sixvertex \draw (3) -- (4) -- (5) -- (6) -- (1) -- (4) -- (2) -- (6) -- (4);\end{tikzpicture}}
\addcase\ngraph{6.82}{\begin{tikzpicture}[smallgraph]\sixvertex \draw (2) -- (4) -- (1) -- (6) -- (5) -- (4) -- (3); \draw (4) -- (6); \draw (1) -- (5);\end{tikzpicture}}
\addcase\ngraph{6.83}{\begin{tikzpicture}[smallgraph]\sixvertex \draw (1) -- (5) -- (6) -- (1) -- (2) -- (3) -- (4) -- (2); \draw (4) -- (5);\end{tikzpicture}}
\addcase\ngraph{6.84}{\begin{tikzpicture}[smallgraph]\sixvertex \draw (1) -- (2) -- (3) -- (4) -- (5) -- (6) -- (1) -- (3); \draw (2) -- (5);\end{tikzpicture}}
\addcase\ngraph{6.85}{\begin{tikzpicture}[smallgraph]\sixvertex \draw (1) -- (2) -- (3) -- (4) -- (5) -- (6) -- (1) -- (3); \draw (2) -- (6);\end{tikzpicture}}
\addcase\ngraph{6.86}{\begin{tikzpicture}[smallgraph]\sixvertex \draw (3) -- (4) -- (5) -- (6) -- (1) -- (2) -- (4) -- (6); \draw (1) -- (5);\end{tikzpicture}}
\addcase\ngraph{6.87}{\begin{tikzpicture}[smallgraph]\sixvertex \draw (1) -- (3) -- (5) -- (4) -- (3) -- (2) -- (6) -- (1); \draw (5) -- (6);\end{tikzpicture}}
\addcase\ngraph{6.88}{\begin{tikzpicture}[smallgraph]\sixvertex \draw (3) -- (4) -- (5) -- (6) -- (1) -- (2) -- (4) -- (1); \draw (2) -- (5);\end{tikzpicture}}
\addcase\ngraph{6.89}{\begin{tikzpicture}[smallgraph]\sixvertex \draw (3) -- (4) -- (5) -- (6) -- (1) -- (2) -- (4) -- (1) -- (5);\end{tikzpicture}}
\addcase\ngraph{6.90}{\begin{tikzpicture}[smallgraph]\sixvertex \draw (4) -- (1) -- (5) -- (6) -- (1) -- (2) -- (3) -- (4) -- (5);\end{tikzpicture}}
\addcase\ngraph{6.91}{\begin{tikzpicture}[smallgraph]\sixvertex \draw (3) -- (2) -- (1) -- (6) -- (5) -- (4) -- (1); \draw (2) -- (5); \draw (4) -- (6);\end{tikzpicture}}
\addcase\ngraph{6.92}{\begin{tikzpicture}[smallgraph]\sixvertex \draw (1) -- (4) -- (3) -- (2) -- (4) -- (5) -- (6) -- (1); \draw (1) -- (5);\end{tikzpicture}}
\addcase\ngraph{6.93}{\begin{tikzpicture}[smallgraph]\sixvertex \draw (6) -- (1) -- (4) -- (3) -- (6) -- (2) -- (4) -- (5) -- (6);\end{tikzpicture}}
\addcase\ngraph{6.94}{\begin{tikzpicture}[smallgraph]\sixvertex \draw (3) -- (1) -- (5) -- (6) -- (1) -- (2) -- (3) -- (4) -- (5);\end{tikzpicture}}
\addcase\ngraph{6.95}{\begin{tikzpicture}[smallgraph]\sixvertex \draw (3) -- (2) -- (1) -- (6) -- (5) -- (4) -- (1) -- (5); \draw (4) -- (6);\end{tikzpicture}}
\addcase\ngraph{6.96}{\begin{tikzpicture}[smallgraph]\sixvertex \draw (1) -- (2) -- (4) -- (5) -- (6) -- (1) -- (4); \draw (2) -- (5) -- (1);\end{tikzpicture}}
\addcase\ngraph{6.97}{\begin{tikzpicture}[smallgraph]\sixvertex \draw (1) -- (2) -- (4) -- (1) -- (5) -- (4); \draw (2) -- (6) -- (1); \draw (5) -- (6);\end{tikzpicture}}
\addcase\ngraph{6.98}{\begin{tikzpicture}[smallgraph]\sixvertex \draw (3) -- (2) -- (1) -- (6) -- (5) -- (4) -- (1) -- (5); \draw (2) -- (4);\end{tikzpicture}}
\addcase\ngraph{6.99}{\begin{tikzpicture}[smallgraph]\sixvertex \draw (6) -- (1) -- (3) -- (2) -- (6) -- (3) -- (4) -- (5) -- (6);\end{tikzpicture}}
\addcase\ngraph{6.100}{\begin{tikzpicture}[smallgraph]\sixvertex \draw (6) -- (5) -- (4) -- (3); \draw (1) -- (5) -- (2) -- (1) -- (4) -- (2);\end{tikzpicture}}
\addcase\ngraph{6.101}{\begin{tikzpicture}[smallgraph]\sixvertex \draw (1) -- (2) -- (3) -- (4) -- (5) -- (6) -- (1) -- (4); \draw (2) -- (5);\end{tikzpicture}}
\addcase\ngraph{6.102}{\begin{tikzpicture}[smallgraph]\sixvertex \draw (3) -- (4) -- (5) -- (6) -- (1) -- (2) -- (5) -- (1) -- (4);\end{tikzpicture}}
\addcase\ngraph{6.103}{\begin{tikzpicture}[smallgraph]\sixvertex \draw (6) -- (1) -- (3) -- (2) -- (6) -- (3) -- (4) -- (6) -- (5) -- (3);\end{tikzpicture}}
\addcase\ngraph{6.104}{\begin{tikzpicture}[smallgraph]\sixvertex \draw (1) -- (2) -- (3) -- (4) -- (5) -- (6) -- (1) -- (3); \draw (4) -- (1) -- (5);\end{tikzpicture}}
\addcase\ngraph{6.105}{\begin{tikzpicture}[smallgraph]\sixvertex \draw (3) -- (4) -- (5) -- (6) -- (1) -- (2) -- (4) -- (1); \draw (2) -- (5); \draw (4) -- (6);\end{tikzpicture}}
\addcase\ngraph{6.106}{\begin{tikzpicture}[smallgraph]\sixvertex \draw (3) -- (4) -- (5) -- (6) -- (1) -- (2) -- (4) -- (1) -- (5); \draw (4) -- (6);\end{tikzpicture}}
\addcase\ngraph{6.107}{\begin{tikzpicture}[smallgraph]\sixvertex \draw (6) -- (1) -- (3) -- (2) -- (6) -- (3) -- (4) -- (5) -- (6); \draw (3) -- (5);\end{tikzpicture}}
\addcase\ngraph{6.108}{\begin{tikzpicture}[smallgraph]\sixvertex \draw (4) -- (2) -- (3) -- (4) -- (1) -- (6) -- (4) -- (5) -- (6); \draw (1) -- (5);\end{tikzpicture}}
\addcase\ngraph{6.109}{\begin{tikzpicture}[smallgraph]\sixvertex \draw (1) -- (2) -- (3) -- (4) -- (5) -- (6) -- (1) -- (5); \draw (2) -- (4); \draw (3) -- (6);\end{tikzpicture}}
\addcase\ngraph{6.110}{\begin{tikzpicture}[smallgraph]\sixvertex \draw (2) -- (3) -- (4) -- (5) -- (6) -- (1) -- (2) -- (6); \draw (5) -- (1) -- (4);\end{tikzpicture}}
\addcase\ngraph{6.111}{\begin{tikzpicture}[smallgraph]\sixvertex \draw (1) -- (2) -- (3) -- (4) -- (5) -- (6) -- (1) -- (5); \draw (6) -- (2) -- (4);\end{tikzpicture}}
\addcase\ngraph{6.112}{\begin{tikzpicture}[smallgraph]\sixvertex \draw (5) -- (6) -- (1) -- (2) -- (3) -- (4) -- (5) -- (1) -- (4); \draw (3) -- (6);\end{tikzpicture}}
\addcase\ngraph{6.113}{\begin{tikzpicture}[smallgraph]\sixvertex \draw (1) -- (2) -- (3) -- (4) -- (5) -- (6) -- (1) -- (5) -- (2) -- (4);\end{tikzpicture}}
\addcase\ngraph{6.114}{\begin{tikzpicture}[smallgraph]\sixvertex \draw (3) -- (4) -- (5) -- (6) -- (1) -- (2) -- (6); \draw (1) -- (5) -- (2) -- (4);\end{tikzpicture}}
\addcase\ngraph{6.115}{\begin{tikzpicture}[smallgraph]\sixvertex \draw (1) -- (2) -- (3) -- (4) -- (5) -- (6) -- (1) -- (5) -- (3) -- (6);\end{tikzpicture}}
\addcase\ngraph{6.116}{\begin{tikzpicture}[smallgraph]\sixvertex \draw (3) -- (4) -- (5) -- (6) -- (1) -- (2) -- (4) -- (6); \draw (1) -- (5) -- (2);\end{tikzpicture}}
\addcase\ngraph{6.117}{\begin{tikzpicture}[smallgraph]\sixvertex \draw (1) -- (2) -- (4) -- (5) -- (6) -- (1) -- (5) -- (2) -- (6) -- (4);\end{tikzpicture}}
\addcase\ngraph{6.118}{\begin{tikzpicture}[smallgraph]\sixvertex \draw (1) -- (2) -- (3) -- (4) -- (5) -- (6) -- (1) -- (5) -- (2) -- (6);\end{tikzpicture}}
\addcase\ngraph{6.119}{\begin{tikzpicture}[smallgraph]\sixvertex \draw (1) -- (2) -- (3) -- (4) -- (5) -- (6) -- (1) -- (4); \draw (2) -- (5); \draw (3) -- (6);\end{tikzpicture}}
\addcase\ngraph{6.120}{\begin{tikzpicture}[smallgraph]\sixvertex \draw (2) -- (3) -- (4) -- (5) -- (6) -- (1) -- (2) -- (4) -- (6) -- (2);\end{tikzpicture}}
\addcase\ngraph{6.121}{\begin{tikzpicture}[smallgraph]\sixvertex \draw (4) -- (5) -- (6) -- (1) -- (3) -- (5); \draw (2) -- (3) -- (4) -- (6) -- (2);\end{tikzpicture}}
\addcase\ngraph{6.122}{\begin{tikzpicture}[smallgraph]\sixvertex \draw (1) -- (2) -- (3) -- (4) -- (5) -- (6) -- (1) -- (4) -- (2) -- (5);\end{tikzpicture}}
\addcase\ngraph{6.123}{\begin{tikzpicture}[smallgraph]\sixvertex \draw (3) -- (4) -- (5) -- (6) -- (1) -- (2) -- (4) -- (1) -- (5) -- (2);\end{tikzpicture}}
\addcase\ngraph{6.124}{\begin{tikzpicture}[smallgraph]\sixvertex \draw (1) -- (2) -- (3) -- (4) -- (5) -- (6) -- (1) -- (4) -- (2) -- (6) -- (4);\end{tikzpicture}}
\addcase\ngraph{6.125}{\begin{tikzpicture}[smallgraph]\sixvertex \draw (1) -- (2) -- (3) -- (4) -- (5) -- (6) -- (1) -- (4) -- (2); \draw (3) -- (5); \draw (4) -- (6);\end{tikzpicture}}
\addcase\ngraph{6.126}{\begin{tikzpicture}[smallgraph]\sixvertex \draw (1) -- (2) -- (3) -- (4) -- (5) -- (6) -- (1) -- (4) -- (2) -- (5); \draw (4) -- (6);\end{tikzpicture}}
\addcase\ngraph{6.127}{\begin{tikzpicture}[smallgraph]\sixvertex \draw (6) -- (1) -- (2) -- (3) -- (4) -- (5) -- (6) -- (4) -- (1) -- (5); \draw (2) -- (4);\end{tikzpicture}}
\addcase\ngraph{6.128}{\begin{tikzpicture}[smallgraph]\sixvertex \draw (6) -- (1) -- (3) -- (2) -- (6) -- (3) -- (4) -- (5) -- (6) -- (4); \draw (3) -- (5);\end{tikzpicture}}
\addcase\ngraph{6.129}{\begin{tikzpicture}[smallgraph]\sixvertex \draw (3) -- (4) -- (5) -- (6) -- (1) -- (2) -- (4) -- (1) -- (5) -- (2); \draw (4) -- (6);\end{tikzpicture}}
\addcase\ngraph{6.130}{\begin{tikzpicture}[smallgraph]\sixvertex \draw (1) -- (2) -- (3) -- (4) -- (5) -- (6) -- (1) -- (3) -- (5); \draw (2) -- (4) -- (6);\end{tikzpicture}}
\addcase\ngraph{6.131}{\begin{tikzpicture}[smallgraph]\sixvertex \draw (1) -- (2) -- (3) -- (4) -- (5) -- (6) -- (1) -- (4) -- (6) -- (3) -- (5);\end{tikzpicture}}
\addcase\ngraph{6.132}{\begin{tikzpicture}[smallgraph]\sixvertex \draw (3) -- (4) -- (5) -- (6) -- (1) -- (2) -- (4) -- (6) -- (2) -- (5) -- (1);\end{tikzpicture}}
\addcase\ngraph{6.133}{\begin{tikzpicture}[smallgraph]\sixvertex \draw (1) -- (2) -- (3) -- (4) -- (5) -- (6) -- (1) -- (3); \draw (2) -- (4) -- (6) -- (2);\end{tikzpicture}}
\addcase\ngraph{6.134}{\begin{tikzpicture}[smallgraph]\sixvertex \draw (1) -- (2) -- (3) -- (4) -- (5) -- (6) -- (1) -- (4) -- (2) -- (5) -- (1);\end{tikzpicture}}
\addcase\ngraph{6.135}{\begin{tikzpicture}[smallgraph]\sixvertex \draw (1) -- (2) -- (3) -- (4) -- (5) -- (6) -- (1) -- (4) -- (2) -- (5); \draw (3) -- (6);\end{tikzpicture}}
\addcase\ngraph{6.136}{\begin{tikzpicture}[smallgraph]\sixvertex \draw (1) -- (2) -- (4) -- (5) -- (6) -- (1) -- (4) -- (6) -- (2) -- (5) -- (1);\end{tikzpicture}}
\addcase\ngraph{6.137}{\begin{tikzpicture}[smallgraph]\sixvertex \draw (1) -- (2) -- (3) -- (4) -- (5) -- (6) -- (1) -- (3) -- (5); \draw (2) -- (6) -- (4);\end{tikzpicture}}
\addcase\ngraph{6.138}{\begin{tikzpicture}[smallgraph]\sixvertex \draw (1) -- (2) -- (3) -- (4) -- (5) -- (6) -- (1) -- (4) -- (2) -- (6) -- (3);\end{tikzpicture}}
\addcase\ngraph{6.139}{\begin{tikzpicture}[smallgraph]\sixvertex \draw (1) -- (3) -- (5) -- (6) -- (1) -- (4) -- (6) -- (2) -- (3) -- (4) -- (5); \draw (3) -- (6);\end{tikzpicture}}
\addcase\ngraph{6.140}{\begin{tikzpicture}[smallgraph]\sixvertex \draw (1) -- (2) -- (3) -- (4) -- (5) -- (6) -- (1) -- (3) -- (5); \draw (2) -- (6) -- (4); \draw (3) -- (6);\end{tikzpicture}}
\addcase\ngraph{6.141}{\begin{tikzpicture}[smallgraph]\sixvertex \draw (1) -- (2) -- (3) -- (4) -- (5) -- (6) -- (1) -- (3) -- (5) -- (2); \draw (4) -- (6) -- (3);\end{tikzpicture}}
\addcase\ngraph{6.142}{\begin{tikzpicture}[smallgraph]\sixvertex \draw (3) -- (4) -- (5) -- (6) -- (1) -- (2) -- (4) -- (6) -- (2) -- (5) -- (1) -- (4);\end{tikzpicture}}
\addcase\ngraph{6.143}{\begin{tikzpicture}[smallgraph]\sixvertex \draw (1) -- (3) -- (4) -- (5) -- (6) -- (1) -- (4); \draw (5) -- (1); \draw (5) -- (3) -- (2) -- (6) -- (3);\end{tikzpicture}}
\addcase\ngraph{6.144}{\begin{tikzpicture}[smallgraph]\sixvertex \draw (1) -- (2) -- (3) -- (4) -- (5) -- (6) -- (1) -- (4) -- (6) -- (2) -- (5); \draw (3) -- (6);\end{tikzpicture}}
\addcase\ngraph{6.145}{\begin{tikzpicture}[smallgraph]\sixvertex \draw (1) -- (2) -- (3) -- (4) -- (5) -- (6) -- (1) -- (4) -- (2) -- (5) -- (1); \draw (3) -- (6);\end{tikzpicture}}
\addcase\ngraph{6.146}{\begin{tikzpicture}[smallgraph]\sixvertex \draw (1) -- (2) -- (3) -- (4) -- (5) -- (6) -- (1) -- (3) -- (5) -- (2) -- (6) -- (4);\end{tikzpicture}}
\addcase\ngraph{6.147}{\begin{tikzpicture}[smallgraph]\sixvertex \draw (1) -- (3) -- (4) -- (5) -- (6) -- (1) -- (4) -- (6) -- (2) -- (3) -- (5) -- (1);\end{tikzpicture}}
\addcase\ngraph{6.148}{\begin{tikzpicture}[smallgraph]\sixvertex \draw (2) -- (3) -- (4) -- (5) -- (6) -- (1) -- (2) -- (6) -- (4) -- (1) -- (3) -- (5); \draw (3) -- (6);\end{tikzpicture}}
\addcase\ngraph{6.149}{\begin{tikzpicture}[smallgraph]\sixvertex \draw (1) -- (2) -- (3) -- (4) -- (5) -- (6) -- (1) -- (3) -- (5) -- (1); \draw (2) -- (5); \draw (3) -- (6); \draw (1) -- (4);\end{tikzpicture}}
\addcase\ngraph{6.150}{\begin{tikzpicture}[smallgraph]\sixvertex \draw (1) -- (2) -- (3) -- (4) -- (5) -- (6) -- (1) -- (3) -- (6) -- (4) -- (1) -- (5) -- (3);\end{tikzpicture}}
\addcase\ngraph{6.151}{\begin{tikzpicture}[smallgraph]\sixvertex \draw (1) -- (2) -- (3) -- (4) -- (5) -- (6) -- (1) -- (4) -- (2) -- (5) -- (3) -- (6) -- (4);\end{tikzpicture}}
\addcase\ngraph{6.152}{\begin{tikzpicture}[smallgraph]\sixvertex \draw (1) -- (2) -- (3) -- (4) -- (5) -- (6) -- (1) -- (3) -- (5) -- (1); \draw (2) -- (4) -- (6) -- (2);\end{tikzpicture}}
\addcase\ngraph{6.153}{\begin{tikzpicture}[smallgraph]\sixvertex \draw (1) -- (2) -- (3) -- (4) -- (5) -- (6) -- (1) -- (4) -- (2) -- (6) -- (4); \draw (2) -- (5) -- (3) -- (6);\end{tikzpicture}}
\addcase\ngraph{6.154}{\begin{tikzpicture}[smallgraph]\sixvertex \draw (1) -- (2) -- (3) -- (4) -- (5) -- (6) -- (1) -- (3) -- (5) -- (2) -- (6) -- (4) -- (1); \draw (3) -- (6);\end{tikzpicture}}
\addcase\ngraph{6.155}{\begin{tikzpicture}[smallgraph]\sixvertex \draw (4) -- (1) -- (2) -- (3) -- (4) -- (5) -- (6) -- (1) -- (3) -- (5) -- (2) -- (4) -- (6) -- (3); \draw (2) -- (6);\end{tikzpicture}}
\addcase\ngraph{6.156}{\begin{tikzpicture}[smallgraph]\sixvertex \draw (4) -- (1) -- (2) -- (3) -- (4) -- (5) -- (6) -- (1) -- (3) -- (5) -- (2) -- (4) -- (6) -- (3); \draw (2) -- (6); \draw (1) -- (5);\end{tikzpicture}}

\def\fl[#1]{\left\lfloor\frac{#1}{2}\right\rfloor}

 \newtheorem{theorem}{Theorem}[section]

 \theoremstyle{definition}
 
 \theoremstyle{remark}
 \newtheorem{remark}[theorem]{Remark}
 
 \numberwithin{equation}{section}

\begin{document}

%
%
%
%
%
%
%
%
%

\title{\bf On the Crossing Numbers of Cartesian Products of Small Graphs with Paths, Cycles and Stars}

%
\author{\Large Kieran Clancy \\
1284 South Road\\
Tonsley 5042\\
Australia \\
%
%
\and
\Large Michael Haythorpe \thanks{Corresponding author}\\
1284 South Road\\
Tonsley 5042\\
Australia\\
{\tt michael.haythorpe@flinders.edu.au}
\and
Alex Newcombe \\
1284 South Road\\
Tonsley 5042\\
Australia\\
{\tt alex.newcombe@flinders.edu.au}}
%
%
\date{February 27, 2019}

\maketitle

\begin{abstract}
There has been significant research dedicated towards computing the crossing numbers of families of graphs resulting from the Cartesian products of small graphs with arbitrarily large paths, cycles and stars. For graphs with four or fewer vertices, these have all been computed, but there are still various gaps for graphs with five or more vertices. We contribute to this field by determining the crossing numbers for fifteen such families.
\end{abstract}

\section{Introduction}
Consider a graph $G$ comprising vertices $V(G)$ and edges $E(G)$. A {\em drawing} of $G$ onto the plane is a mapping $D_G$ which maps vertices to distinct points and edges to continuous arcs. The image of an edge $e = \{u,v\}$ is a continuous arc between the points associated with $u$ and $v$ such that the interior of the arc does not contain any points associated with vertices. Without loss of generality, we call the points and arcs the `vertices' and `edges' of the drawing. The interiors of the edges are allowed to intersect at singleton points in such a way that each edge intersection is strictly a crossing between the edges, as opposed to the edges touching and then not crossing. For simplicity, we assume that no three edges intersect at the same point. These intersections form the {\em crossings} of the drawing and the number of crossings in $D$ is denoted by $cr_{D}(G)$. The {\em crossing number} of $G$, denoted $cr(G)$, is the minimum number of crossings over all possible drawings. The {\em crossing number problem} (CNP) is the problem of determining the crossing number of a graph, and is known to be NP-hard \cite{gareyjohnson1983}. The CNP is a notoriously difficult problem even for relatively small graphs; indeed, the crossing number of $K_{13}$ has still not been determined \cite{mcquillanetal2015}.

The Cartesian product of two graphs $G$ and $H$, denoted by $G \Box H$, is a graph with vertex set $V(G) \times V(H)$, such that an edge exists between vertices $(u,u')$ and $(v,v')$ if and only if either $u = v$ and $\{u',v'\} \in E(H)$, or $u' = v'$ and $\{u,v\} \in E(G)$. An example of the Cartesian product of two paths, $P_3 \Box P_4$, is displayed in Figure \ref{fig-p3boxp4}. Note that $P_n$ is the path on $n+1$ vertices.

\begin{figure}[h!]\begin{center}
\begin{tikzpicture}[largegraph,scale=0.35]
\foreach \n in {0,...,4}{
  \node (a\n) at (2*\n,6) {};
  \node (b\n) at (2*\n,4) {};
  \node (c\n) at (2*\n,2) {};
  \node (d\n) at (2*\n,0) {};
}
\foreach \n in {0,...,3}{
  \pgfmathtruncatemacro{\m}{\n + 1}
  \draw (a\n) -- (a\m);
  \draw (b\n) -- (b\m);
  \draw (c\n) -- (c\m);
  \draw (d\n) -- (d\m);
}
\foreach \n in {0,...,4}{
  \draw (a\n) -- (b\n);
  \draw (b\n) -- (c\n);
  \draw (c\n) -- (d\n);
}

\end{tikzpicture}\caption{The Cartesian product $P_3 \Box P_4$.\label{fig-p3boxp4}}\end{center}\end{figure}

One of the early results relating to crossing numbers is due to Beineke and Ringeisen \cite{beinekeringeisen1980} who, in 1980, considered families of graphs resulting from the Cartesian products of connected graphs on four vertices with arbitrarily large cycles. There are six connected graphs on four vertices, and with only one exception (the star $S_3$), they were successful in determining the crossing numbers for each resulting family. The one unsolved case was subsequently handled by Jendrol and \v{S}cerbov\'{a} \cite{jendrolscerbova1982} in 1982. A decade later in 1994, Kle\v{s}\v{c} \cite{klesc1994} extended this result by determining the crossing numbers of families resulting from the Cartesian products of each of the connected graphs on four vertices with arbitrarily large paths and stars. In the ensuing years, significant effort has gone into extending these results to include graphs on more vertices; in particular five and six vertices. The pioneering work in this area was by Kle\v{s}\v{c} and his various co-authors \cite{drazenska2011,drazenskaklesc2008,drazenskaklesc2011,klesc1991,klesc1994,klesc1995,klesc1996,klesc1999,klesc1999_2,klesc2001,klesc2001_2,klesc2002,klesc2005,klesc2009,klesckocurova2007,klesckravecova2008,klesckravecova2012,klescetal2014,klescpetrillova2013,klescpetrillova2013_2,klescetal2017,klescetal1996,klescschrotter2013}
 who have spent the last three decades handling these cases, often on a graph-by-graph basis, requiring ad-hoc proofs that exploit the specific graph structure of the graphs in question. In the last fifteen years, a large number of other researchers have also contributed to this field. However, communication between the various researchers in this area has been poor, and it is has not been uncommon for multiple researchers to publish identical results.

To address this issue, a dynamic survey \cite{survey} on graphs with known crossing numbers was recently produced, which included tables of all known results of crossing numbers of families resulting from Cartesian products of small graphs with paths, cycles and stars. We reproduce the tables for crossing numbers of Cartesian products involving graphs on six vertices here. They are separated into Cartesian products involving paths (Table \ref{tab-6vcart_path}), cycles (Table \ref{tab-6vcart_cycle}) and stars (Table \ref{tab-6vcart_star}). In Tables \ref{tab-6vcart_path}--\ref{tab-6vcart_star}, only those graphs for which results have been determined are included. The graph indices are taken from Harary \cite{harary1969}, and an illustration of each graph on six vertices, as well as citations for each of the results in Tables \ref{tab-6vcart_path}--\ref{tab-6vcart_star} may be found in \cite{survey}. Note that, up to isomorphism, there are 156 graphs on six vertices, which includes 112 connected graphs.

For completeness, the tables include results determined in this paper and we highlight these indices in boldface. Additionally, note that some of the results are marked with asterisks. This indicates that the corresponding results appeared in journals which do not have adequate peer review processes. Hence, it would be valuable if these results could be re-proved in a journal which is fully peer reviewed. Indeed, in the present work, we provide a proof for one such asterisked result from \cite{su1}, namely $G^6_{120} \Box P_n$.

Proving that a particular graph family has crossing number equal to a particular function is usually achieved as follows. First, an upper bound for the crossing number is determined by providing a drawing method for members of that family which realises the proposed number of crossings. This is then shown to coincide with a lower bound, which is usually determined by some form of inductive argument. The latter typically takes much more work than the former. However, in some cases, a lower bound can be easily determined. For instance, consider $G^6_{46}$ and $G^6_{60}$, which are displayed in Table \ref{tab-6vcart_path}. It is clear that the former is a subgraph of the latter. Then, for any graph $H$, it follows from the definition of the Cartesian product that $G^6_{46} \Box H$ will be a subgraph of $G^6_{60} \Box H$. Thus, any lower bound for the crossing number of the former also provides a lower bound for the crossing number of the latter.

Furthermore, it is also clear that $G^6_{46}$ contains a subgraph $F$ consisting of a triangle with one pendant vertex attached. Then, any lower bound for $cr(F \Box H)$ also serves as a lower bound for $cr(G^6_{46} \Box H)$. Beineke and Ringeisen \cite{beinekeringeisen1980} proved that $cr(F \Box P_n) = n-1$, and, thus, a corollary of the above arguments implies that $cr(G^6_{60} \Box P_n) \geq cr(G^6_{46} \Box P_n) \geq n-1$. Hence, simply providing a drawing which establishes that $cr(G^6_{60} \Box P_n) \leq n-1$ is sufficient to decide the cases for both $G^6_{46} \Box P_n$ and $G^6_{60} \Box P_n$; indeed, this exact argument was used in Kle\v{s}\v{c} and Petrillov\'{a} \cite{klescpetrillova2013_2} to determine the crossing number of $G^6_{46} \Box P_n$. Of course, this kind of approach is only useful when the upper bound coincides with an established lower bound for a subgraph.

In what follows, we use approaches similar to the previous paragraph to determine the crossing number for fifteen additional families of graphs. Although the arguments are not complicated, the extensive research into filling in the gaps of Tables \ref{tab-6vcart_path} -- \ref{tab-6vcart_star}, which continues to this day, indicates the interest in this area; despite all of that research, these results have been hitherto undiscovered. We are in a unique position to present these simple arguments for two reasons. First, we are able to take advantage of the recently produced dynamic survey \cite{survey} that gathers, for the first time, all known published results into one place, so that they can all be simultaneously drawn upon to provide good lower bounds. Second, we are also able to take advantage of the recently developed crossing minimisation heuristic, QuickCross \cite{quickcross}, to aid us in finding good upper bounds.

\begin{table}[htbp]\begin{center}\scalebox{0.85}{$\begin{array}{|c|c|c||c|c|c||c|c|c|}\hline
\rule{0pt}{2.3ex} i & G^6_i 	& cr(G^6_i \Box P_n) 	&i & G^6_i 	& cr(G^6_i \Box P_n)  &i & G^6_i 	& cr(G^6_i \Box P_n)  	\\
\hline 25&\ctikz{\ngraph{6.25}}  	& 0	    	& 60& \ctikz{\ngraph{6.60}}  	& n - 1     & 89                & \ctikz{\ngraph{6.89}}  	& 3n - 3 	    \\
\hline 26&\ctikz{\ngraph{6.26}}  	& n-1    	& 61& \ctikz{\ngraph{6.61}}  	& 2n        &\text{{\bf 90}}    & \ctikz{\ngraph{6.90}}  	& 3n - 3 		\\        	
\hline 27&\ctikz{\ngraph{6.27}}  	& 2n - 2  	& 64& \ctikz{\ngraph{6.64}}  	& 2n - 2    &91                 & \ctikz{\ngraph{6.91}}  	& 3n - 1 		\\        	
\hline 28&\ctikz{\ngraph{6.28}}  	& n - 1   	& 65& \ctikz{\ngraph{6.65}}  	& 3n - 3    &93                 &\ctikz{\ngraph{6.93}}  	& \unv{4n} 		\\        	
\hline 29&\ctikz{\ngraph{6.29}}  	& 2n - 2  	& 66& \ctikz{\ngraph{6.66}}  	& 2n - 2    &94                 &\ctikz{\ngraph{6.94}}  	& 2n - 2 		\\        	
\hline 31&\ctikz{\ngraph{6.31}}  	& 4n - 4  	& 68& \ctikz{\ngraph{6.68}}  	& 3n - 1    &103                &\ctikz{\ngraph{6.103}} 	& 6n - 2 		\\        	
\hline 40&\ctikz{\ngraph{6.40}}  	& 0	  		& 70& \ctikz{\ngraph{6.70}}  	& 3n - 3    &104                &\ctikz{\ngraph{6.104}} 	& 4n - 4 		\\        	
\hline 41&\ctikz{\ngraph{6.41}}  	& n - 1  	& 71& \ctikz{\ngraph{6.71}}  	& 3n - 1    &109                &\ctikz{\ngraph{6.109}} 	& \unv{4n}  	\\        	
\hline 42&\ctikz{\ngraph{6.42}}  	& 2n - 4  	& 72& \ctikz{\ngraph{6.72}}  	& 4n - 4    &111                &\ctikz{\ngraph{6.111}} 	& 3n - 1 		\\        	
\hline 43&\ctikz{\ngraph{6.43}}  	& n - 1   	& 73& \ctikz{\ngraph{6.73}}  	& 4n - 4    &113                &\ctikz{\ngraph{6.113}} 	& 4n - 4  		\\        	
\hline 44&\ctikz{\ngraph{6.44}}  	& 2n - 2 	& 74& \ctikz{\ngraph{6.74}}  	& 2n - 2    &119                &\ctikz{\ngraph{6.119}} 	& \unv{7n - 1} 	\\        	
\hline 45&\ctikz{\ngraph{6.45}}  	& 2n - 2  	& 75& \ctikz{\ngraph{6.75}}  	& 2n        &\text{{\bf120}}    &\ctikz{\ngraph{6.120}} 	& 3n - 3        \\        	
\hline 46&\ctikz{\ngraph{6.46}}  	& n - 1   	& 77& \ctikz{\ngraph{6.77}}  	& 2n - 2    &121                &\ctikz{\ngraph{6.121}} 	& \unv{4n}     	\\
\hline 47&\ctikz{\ngraph{6.47}}	    & 2n - 2 	& 79& \ctikz{\ngraph{6.79}}  	& 4n - 4    &125                &\ctikz{\ngraph{6.125}} 	& 5n - 3  		\\
\hline 48&\ctikz{\ngraph{6.48}}	    & 4n - 4  	& 80& \ctikz{\ngraph{6.80}}  	& 4n - 4    &146                &\ctikz{\ngraph{6.146}} 	& \unv{5n - 1} 	\\
\hline 51&\ctikz{\ngraph{6.51}}	    & 3n - 3 	& 83& \ctikz{\ngraph{6.83}}  	& 2n - 2    &152                &\ctikz{\ngraph{6.152}} 	& \unv{6n}   	\\
\hline 53&\ctikz{\ngraph{6.53}}	    & 2n - 2 	& 85& \ctikz{\ngraph{6.85}}  	& 2n        &154                &\ctikz{\ngraph{6.154}} 	& 9n - 1 		\\
\hline 54&\ctikz{\ngraph{6.54}}	    & 2n - 2 	& 86& \ctikz{\ngraph{6.86}}  	& 3n - 1    &155                &\ctikz{\ngraph{6.155}} 	& \unv{12n}    	\\
\hline 59&\ctikz{\ngraph{6.59}}	    & 2n - 2  	& 87& \ctikz{\ngraph{6.87}}  	& 3n - 1    &156                &\ctikz{\ngraph{6.156}} 	& 15n + 3		\\
\hline
\end{array}$}\caption{Known crossing numbers of Cartesian products of graphs on six vertices with paths. All results are for $n \geq 1$ and boldface indices are results derived in this paper.\label{tab-6vcart_path}}\end{center}\end{table}

\begin{table}[htbp]\begin{center}\scalebox{0.825}{$\begin{array}{|c|c|cc||c|c|cc|}
\hline\rule{0pt}{2.3ex} i & G^6_i & cr(G^6_i \Box C_n) & \mbox{(small cases)} & i & G^6_i & cr(G^6_i \Box C_n) & \mbox{(small cases)} \\
\hline 25                 & \ctikz{\ngraph{6.25}}  & \begin{array}{rl}0  	& \end{array}              &                                                                   &\text{{\bf64}}     & \ctikz{\ngraph{6.64}}  & \begin{array}{rl}2n    & (n \geq 6)\end{array}    & \begin{array}{rl}6  & (n=3)\\8  & (n=4)\\10  & (n=5)\end{array}    \\
\hline 40                 & \ctikz{\ngraph{6.40}}  & \begin{array}{rl}4n    & (n \geq 6)\end{array}    & \begin{array}{rl} 6  & (n=3)\\12  & (n=4)\\18  & (n=5)\end{array} &\text{{\bf66}}     & \ctikz{\ngraph{6.66}}  & \begin{array}{rl}3n    & (n \geq 5)\end{array}    & \begin{array}{rl} 7  & (n=3)\\12  & (n=4)\end{array}               \\	
\hline 41                 & \ctikz{\ngraph{6.41}}  & \begin{array}{rl}3n    & (n \geq 5)\end{array}    & \begin{array}{rl} 5  & (n=3)\\10  & (n=4)\end{array}              &67                 & \ctikz{\ngraph{6.67}}  & \begin{array}{rl}3n    & (n \geq 4)\end{array}    & \begin{array}{rl} \\7  & (n=3)\\  & \end{array}                    \\
\hline 42                 & \ctikz{\ngraph{6.42}}  & \begin{array}{rl}2n    & (n \geq 4)\end{array}    & \begin{array}{rl} 4  & (n=3)\end{array}                           &\text{{\bf70}}     & \ctikz{\ngraph{6.70}}  & \begin{array}{rl}3n    & (n \geq 5)\end{array}    & \begin{array}{rl} 7 & (n=3)\\12  & (n=4)\end{array}                \\
\hline 43                 & \ctikz{\ngraph{6.43}}  & \begin{array}{rl}\;\;n & (n \geq 3)\end{array}    &                                                                   &\text{{\bf75}}     & \ctikz{\ngraph{6.75}}  & \begin{array}{rl}2n    & (n \geq 4)\end{array}    & \begin{array}{rl} 6  & (n=3)\end{array}                            \\
\hline 44                 & \ctikz{\ngraph{6.44}}  & \begin{array}{rl}2n    & (n \geq 4)\end{array}    & \begin{array}{rl} 4  & (n=3)\end{array}                           &\text{{\bf77}}     & \ctikz{\ngraph{6.77}}  & \begin{array}{rl}2n    & (n \geq 6)\end{array}    & \begin{array}{rl} 6 & (n=3)\\8  & (n=4)\\10  & (n=5)\end{array}    \\
\hline 46                 & \ctikz{\ngraph{6.46}}  & \begin{array}{rl}\;\;n & (n \geq 3)\end{array}    &                                                                   &78                 & \ctikz{\ngraph{6.78}}  & \begin{array}{rl}3n    & (n \geq 6)\end{array}    & \begin{array}{rl} 7  & (n=3)\\10  & (n=4)\\14  & (n=5)\end{array}  \\
\hline 47                 & \ctikz{\ngraph{6.47}}  & \begin{array}{rl}2n    & (n \geq 6)\end{array}    & \begin{array}{rl} 4  & (n=3)\\6  & (n=4)\\9  & (n=5)\end{array}   &\text{{\bf83}}     & \ctikz{\ngraph{6.83}}  & \begin{array}{rl}4n    & (n \geq 6)\end{array}    & \begin{array}{rl} 10  & (n=3)\\16  & (n=4)\\20  & (n=5)\end{array} \\
\hline 49                 & \ctikz{\ngraph{6.49}}  & \begin{array}{rl}2n    & (n \geq 4)\end{array}    & \begin{array}{rl} 4  & (n=3)\end{array}                           &\text{{\bf90}}     & \ctikz{\ngraph{6.90}}  & \begin{array}{rl}4n    & (n \geq 6)\end{array}    & \begin{array}{rl} 11 & (n=3)\\16  & (n=4)\\20  & (n=5)\end{array}  \\
\hline 53                 & \ctikz{\ngraph{6.53}}  & \begin{array}{rl}2n    & (n \geq 6)\end{array}    & \begin{array}{rl} 4  & (n=3)\\6  & (n=4)\\9  & (n=5)\end{array}   &\text{{\bf92}}     & \ctikz{\ngraph{6.92}}  & \begin{array}{rl}3n    & (n \geq 4)\end{array}    & \begin{array}{rl} 9  & (n=3)\end{array}                            \\
\hline 54                 & \ctikz{\ngraph{6.54}}  & \begin{array}{rl}2n    & (n \geq 6)\end{array}    & \begin{array}{rl} 4  & (n=3)\\6  & (n=4)\\9  & (n=5)\end{array}   &\text{{\bf98}}     & \ctikz{\ngraph{6.98}}  & \begin{array}{rl}3n    & (n \geq 5)\end{array}    & \begin{array}{rl} 9  & (n=3)\\12 & (n=4)\end{array}                \\
\hline \text{{\bf59}}     & \ctikz{\ngraph{6.59}}  & \begin{array}{rl}4n    & (n \geq 6)\end{array}    & \begin{array}{rl} 8  & (n=3)\\16  & (n=4)\\20  & (n=5)\end{array} &113                & \ctikz{\ngraph{6.113}} & \begin{array}{rl}4n    & (n \geq 3)\end{array}    &                                                                    \\
\hline \text{{\bf60}}     & \ctikz{\ngraph{6.60}}  & \begin{array}{rl}4n    & (n \geq 6)\end{array}    & \begin{array}{rl} 8  & (n=3)\\16  & (n=4)\\20  & (n=5)\end{array} &156                & \ctikz{\ngraph{6.156}} & \begin{array}{rl}18n   & (n \geq 3)\end{array}    &                                                                    \\
\hline \text{{\bf63}}     & \ctikz{\ngraph{6.63}}  & \begin{array}{rl}2n    & (n \geq 4)\end{array}    & \begin{array}{rl} 6  & (n=3)\end{array}                           &		            &						 &						 \\
\hline\end{array}$}\caption{Crossing numbers of Cartesian products of graphs on six vertices with cycles. Boldface indices are results derived in this paper\label{tab-6vcart_cycle}}\end{center}\end{table}

\begin{table}[htbp]\begin{center}\scalebox{0.85}{$\begin{array}{|c|c|c||c|c|c|}
\hline\rule{0pt}{2.3ex} i 	               & G^6_i 			        & cr(G^6_i \Box S_n)  		    & i 	& G^6_i 		         & cr(G^6_i \Box S_n) 			        \\
\hline\rule{0pt}{2.3ex}25                  & \ctikz{\ngraph{6.25}}  & 4\fl[n]\fl[n-1]     		    &77     & \ctikz{\ngraph{6.77}}  & 4\fl[n]\fl[n-1] + 2\fl[n]     		\\
\hline \rule{0pt}{2.3ex}26                 & \ctikz{\ngraph{6.26}}  & 4\fl[n]\fl[n-1] + \fl[n]   	&79     & \ctikz{\ngraph{6.79}}  & 6\fl[n]\fl[n-1] + 4\fl[n]      		\\
\hline \rule{0pt}{2.3ex}27                 & \ctikz{\ngraph{6.27}}  & 5\fl[n]\fl[n-1] + 2\fl[n]  	&80     & \ctikz{\ngraph{6.80}}  & 6\fl[n]\fl[n-1] + 4\fl[n]        	\\
\hline \rule{0pt}{2.3ex}28                 & \ctikz{\ngraph{6.28}}  & 4\fl[n]\fl[n-1] + \fl[n]   	&85     & \ctikz{\ngraph{6.85}}  & \unv{6\fl[n]\fl[n-1] + 2n}       	\\
\hline \rule{0pt}{2.3ex}29                 & \ctikz{\ngraph{6.29}}  & 4\fl[n]\fl[n-1] + 2\fl[n]  	&93     & \ctikz{\ngraph{6.93}}  & \unv{6\fl[n]\fl[n-1] + 4n}       	\\
\hline \rule{0pt}{2.3ex}31                 & \ctikz{\ngraph{6.31}}  & 6\fl[n]\fl[n-1] + 4\fl[n]  	&94     & \ctikz{\ngraph{6.94}}  & 6\fl[n]\fl[n-1] + 2\fl[n]   			\\
\hline \rule{0pt}{2.3ex}43                 & \ctikz{\ngraph{6.43}}  & 4\fl[n]\fl[n-1] + \fl[n]  	&104    & \ctikz{\ngraph{6.104}} & 6\fl[n]\fl[n-1] + 4\fl[n]      		\\
\hline 47     		                       & \ctikz{\ngraph{6.47}}  & 5\fl[n]\fl[n-1] + 2\fl[n]  	&111    & \ctikz{\ngraph{6.111}} & \unv{6\fl[n]\fl[n-1] + 2\fl[n] + 2n }\\
\hline \rule{0pt}{2.3ex}48                 & \ctikz{\ngraph{6.48}}  & 6\fl[n]\fl[n-1] + 4\fl[n]  	&120    & \ctikz{\ngraph{6.120}} & \unv{6\fl[n]\fl[n-1] + 3\fl[n]}  	\\
\hline 53     			                   & \ctikz{\ngraph{6.53}}  &4\fl[n]\fl[n-1] + 2\fl[n]   	&124    & \ctikz{\ngraph{6.124}} & \unv{6\fl[n]\fl[n-1] + 2n + 3\fl[n]} \\
\hline \rule{0pt}{2.3ex}59                 & \ctikz{\ngraph{6.59}}  & 6 \fl[n]\fl[n-1] + 2\fl[n] 	&125    & \ctikz{\ngraph{6.125}} & 6\fl[n]\fl[n-1] + 3\fl[n] + 2n   	\\
\hline 61     			                   & \ctikz{\ngraph{6.72}}  & \unv{6\fl[n]\fl[n-1] + 2n} 	&130    & \ctikz{\ngraph{6.130}} & \unv{6\fl[n]\fl[n-1] + 4n }         	\\
\hline \rule{0pt}{2.3ex}\text{\bf{62}}     & \ctikz{\ngraph{6.62}}  & 5\fl[n]\fl[n-1] + 2\fl[n]  	&137    & \ctikz{\ngraph{6.137}} & \unv{6\fl[n]\fl[n-1] + 4n }       	\\
\hline \rule{0pt}{2.3ex}72                 & \ctikz{\ngraph{6.72}}  & 6\fl[n]\fl[n-1] + 4\fl[n] 	&152    & \ctikz{\ngraph{6.152}} & 6\fl[n]\fl[n-1] + 6n        			\\
\hline \rule{0pt}{2.3ex}73                 & \ctikz{\ngraph{6.73}}  & 6\fl[n]\fl[n-1] + 4\fl[n]  	&	    &				         &						                \\
\hline
\end{array}$}\caption{Crossing numbers of Cartesian products of graphs on six vertices with stars. All results are for $n \geq 1$ and boldface indices are results derived in this paper.\label{tab-6vcart_star}}\end{center}\end{table}


\section{New Results}

%

In this section we determine the crossing numbers of the Cartesian product of one of the graphs displayed in Figure \ref{fig-fifteen}, with various arbitrarily large paths, cycles and stars.

The upcoming proofs are laid out as follows. In Theorem \ref{thm-paths} the crossing numbers of $G^6_{90} \Box P_n$ and $G^6_{120} \Box P_n$ are determined. In Theorems \ref{thm-59608390Cn} and \ref{thm-cycles} the crossing numbers of the Cartesian products of various graphs in Figure \ref{fig-fifteen} with cycles are determined. In Theorems \ref{thm-59608390Cn} and \ref{thm-cycles}, the results are only proved for sufficiently large cycles, and so the remaining cases involving small cycles are handled in Remark \ref{rem-small}. Finally, in Theorem \ref{thm-stars} the crossing number of $G^6_{62} \Box S_n$ is determined. In all cases the lower bounds are obtained from previously published results. In Theorem \ref{thm-59608390Cn} the upper bounds are also obtained from previously published results, and for the other theorems they are established by figures which show drawing methods for each Cartesian product considered.

\begin{figure}[htbp!]\begin{tabular}{ccccccc}
\ctikz{\ngraph{6.59}} & \ctikz{\ngraph{6.60}} & \ctikz{\ngraph{6.62}} & \ctikz{\ngraph{6.63}} & \ctikz{\ngraph{6.64}} & \ctikz{\ngraph{6.66}} & \ctikz{\ngraph{6.70}}  \\
{\small $G^6_{59}$} & {\small $G^6_{60}$} & {\small $G^6_{62}$} & {\small $G^6_{63}$} & {\small $G^6_{64}$} & {\small $G^6_{66}$} & {\small $G^6_{70}$}  \\
\ctikz{\ngraph{6.75}} & \ctikz{\ngraph{6.77}} & \ctikz{\ngraph{6.83}} & \ctikz{\ngraph{6.90}} & \ctikz{\ngraph{6.92}} & \ctikz{\ngraph{6.98}} & \ctikz{\ngraph{6.120}} \\
{\small $G^6_{75}$} & {\small $G^6_{77}$} & {\small $G^6_{83}$} & {\small $G^6_{90}$} & {\small $G^6_{92}$} & {\small $G^6_{98}$} & {\small $G^6_{120}$} \\
\end{tabular}\caption{We will derive new results for Cartesian products involving these fourteen graphs.\label{fig-fifteen}}\end{figure}

\begin{theorem}Consider the path graph $P_n$ for $n \geq 1$. Then, $cr(G^6_{90} \Box P_n) = cr(G^6_{120} \Box P_n) = 3n-3$.\label{thm-paths}\end{theorem}

\begin{proof}Consider the graph $G^6_{51}$ which is displayed in Table \ref{tab-6vcart_path}. The crossing number $cr(G^6_{51} \Box P_n) = 3n-3$ for $n \geq 1$ was determined by Kle\v{s}\v{c} et al \cite{klescetal2014}. It is clear that $G^6_{51}$ is a subgraph of both $G^6_{90}$ and $G^6_{120}$. Hence, we have $cr(G^6_{90} \Box P_n) \geq 3n-3$ for $n \geq 1$, and $cr(G^6_{120} \Box P_n) \geq 3n-3$ for $n \geq 1$. It can be verified that the drawing methods for $G^6_{90} \Box P_n$ displayed in Figure \ref{fig-90Pn} and $G^6_{120} \Box P_n$ displayed in Figure \ref{fig-120Pn} each realise precisely $3n-3$ crossings, completing the proof.\end{proof}

\begin{theorem}Consider the cycle graph $C_n$ for $n \geq 6$. Then, $cr(G^6_{59} \Box C_n) = cr(G^6_{60} \Box C_n) = cr(G^6_{83} \Box C_n) = cr(G^6_{90} \Box C_n) = 4n$.\label{thm-59608390Cn}\end{theorem}

\begin{proof}Consider graphs $G^6_{40}$ and $G^6_{113}$, displayed in Table \ref{tab-6vcart_cycle}. The crossing number $cr(G^6_{40} \Box C_n) = 4n$ for $n \geq 6$ was determined by Richter and Salazar \cite{richtersalazar2001}, and the crossing number $cr(G^6_{113} \Box C_n) = 4n$ for $n \geq 3$ was determined by Kle\v{s}\v{c} and Kravecov\'{a} \cite{klesckravecova2008}. Then, consider graphs $G^6_{59}$, $G^6_{60}$, $G^6_{83}$ and $G^6_{90}$. It is clear that $G^6_{40}$ is a subgraph of each of them, and $G^6_{113}$ is a supergraph of each of them. The result follows immediately.\end{proof}

\begin{figure}[h!]\begin{center}\begin{tikzpicture}[thick,scale=0.45]
\draw (0,0) circle (1.5cm);
\draw (0,0) circle (2.5cm);
\draw (0,0) circle (3.5cm);
\draw (0,0) circle (5.5cm);
\draw (0,0) circle (6.5cm);
\foreach \a in {1,2,...,6}{
\draw (\a*360/6: 1.5cm) node{};
}
\foreach \a in {1,2,...,6}{
\draw (\a*360/6: 2.5cm) node{};
}
\foreach \a in {1,2,...,6}{
\draw (\a*360/6: 3.5cm) node{};
}
\foreach \a in {1,2,...,6}{
\draw (\a*360/6: 5.5cm) node{};
\draw[densely dotted] (\a*360/6: 4.2cm) -- (\a*360/6: 1.5cm);
\draw[densely dotted] (\a*360/6: 6.5cm) -- (\a*360/6: 4.9cm);
\begin{scope}[rotate=\a*360/6]
\node[draw=none,fill=none,transform shape] () at (0:4.5cm) {$\dots$};
\end{scope}
}
\foreach \a in {1,2,...,6}{
\draw (\a*360/6: 6.5cm) node{};
}
\draw (360/6: 1.5cm) to[out=180,in=60] (3*360/6: 1.5cm);
\draw (360/6: 1.5cm) to[out=190,in=120] (4*360/6: 1.5cm);
\draw (1*360/6: 2.5cm) to[out=165,in=30] (2*360/6: 2.15cm);
\draw (2*360/6: 2.15cm) to[out=210,in=75] (3*360/6: 2.5cm);
\draw (1*360/6: 2.5cm) to[out=175,in=30] (2*360/6: 1.95cm);
\draw (2*360/6: 1.95cm) to[out=210,in=90] (3*360/6: 2cm);
\draw (3*360/6: 2cm) to[out=270,in=125] (4*360/6: 2.5cm);
\draw (1*360/6: 3.5cm) to[out=165,in=30] (2*360/6: 3.15cm);
\draw (2*360/6: 3.15cm) to[out=210,in=75] (3*360/6: 3.5cm);
\draw (1*360/6: 3.5cm) to[out=175,in=30] (2*360/6: 2.95cm);
\draw (2*360/6: 2.95cm) to[out=210,in=90] (3*360/6: 3cm);
\draw (3*360/6: 3cm) to[out=270,in=130] (4*360/6: 3.5cm);
\draw (1*360/6: 5.5cm) to[out=165,in=30] (2*360/6: 5.1cm);
\draw (2*360/6: 5.1cm) to[out=210,in=80] (3*360/6: 5.5cm);
\draw (1*360/6: 5.5cm) to[out=175,in=30] (2*360/6: 4.9cm);
\draw (2*360/6: 4.9cm) to[out=210,in=90] (3*360/6: 5cm);
\draw (3*360/6: 5cm) to[out=270,in=135] (4*360/6: 5.5cm);
\draw (1*360/6: 6.5cm) to[out=140,in=30] (2*360/6: 7cm);
\draw (2*360/6: 7cm) to[out=210,in=100] (3*360/6: 6.5cm);
\draw (1*360/6: 6.5cm) to[out=130,in=30] (2*360/6: 7.2cm);
\draw (2*360/6: 7.2cm) to[out=210,in=90] (3*360/6: 7cm);
\draw (3*360/6: 7cm) to[out=270,in=160] (4*360/6: 6.5cm);
\end{tikzpicture}\caption{A drawing of $G^6_{90} \Box P_n$ with $3n-3$ crossings. Each circle of vertices is one copy of $G^6_{90}$.\label{fig-90Pn}}\end{center}\end{figure}

\begin{figure}[h!]\begin{center}\begin{tikzpicture}[thick,scale=0.45]
\draw (0,0) circle (1.5cm);
\draw (0,0) circle (2.5cm);
\draw (0,0) circle (3.5cm);
\draw (0,0) circle (5.5cm);
\draw (0,0) circle (6.5cm);
\foreach \a in {1,2,...,6}{
\draw (\a*360/6: 1.5cm) node{};
}
\foreach \a in {1,2,...,6}{
\draw (\a*360/6: 2.5cm) node{};
}
\foreach \a in {1,2,...,6}{
\draw (\a*360/6: 3.5cm) node{};
}
\foreach \a in {1,2,...,6}{
\draw (\a*360/6: 5.5cm) node{};
\draw[densely dotted] (\a*360/6: 4.2cm) -- (\a*360/6: 1.5cm);
\draw[densely dotted] (\a*360/6: 6.5cm) -- (\a*360/6: 4.9cm);
\begin{scope}[rotate=\a*360/6]
\node[draw=none,fill=none,transform shape] () at (0:4.5cm) {$\dots$};
\end{scope}
}
\foreach \a in {1,2,...,6}{
\draw (\a*360/6: 6.5cm) node{};
}
\draw (360/6: 1.5cm) to[out=190,in=45] (3*360/6: 1.5cm);
\draw (360/6: 1.5cm) to[out=285,in=70] (5*360/6: 1.5cm);
\draw (3*360/6: 1.5cm) to[out=-45,in=170] (5*360/6: 1.5cm);
\draw (1*360/6: 2.5cm) to[out=180,in=60] (3*360/6: 2.5cm);
\draw (1*360/6: 2.5cm) to[out=300,in=60] (5*360/6: 2.5cm);
\draw (3*360/6: 2.5cm) to[out=-60,in=180] (5*360/6: 2.5cm);
\draw (1*360/6: 3.5cm) to[out=310,in=90] (6*360/6: 3cm);
\draw (6*360/6: 3cm) to[out=270,in=50] (5*360/6: 3.5cm);
\draw (1*360/6: 3.5cm) to[out=165,in=30] (2*360/6: 3cm);
\draw (2*360/6: 3cm) to[out=210,in=70] (3*360/6: 3.5cm);
\draw (3*360/6: 3.5cm) to[out=290,in=150] (4*360/6: 3cm);
\draw (4*360/6: 3cm) to[out=-30,in=195] (5*360/6: 3.5cm);
\draw (1*360/6: 5.5cm) to[out=-45,in=90] (6*360/6: 5.1cm);
\draw (6*360/6: 5.1cm) to[out=270,in=40] (5*360/6: 5.5cm);
\draw (1*360/6: 5.5cm) to[out=160,in=30] (2*360/6: 5.1cm);
\draw (2*360/6: 5.1cm) to[out=210,in=80] (3*360/6: 5.5cm);
\draw (3*360/6: 5.5cm) to[out=280,in=150] (4*360/6: 5.1cm);
\draw (4*360/6: 5.1cm) to[out=-30,in=200] (5*360/6: 5.5cm);
\draw (1*360/6: 6.5cm) to[out=140,in=30] (2*360/6: 7cm);
\draw (2*360/6: 7cm) to[out=210,in=100] (3*360/6: 6.5cm);
\draw (1*360/6: 6.5cm) to[out=340,in=90] (6*360/6: 7cm);
\draw (6*360/6: 7cm) to[out=270,in=20] (5*360/6: 6.5cm);
\draw (3*360/6: 6.5cm) to[out=260,in=150] (4*360/6: 7cm);
\draw (4*360/6: 7cm) to[out=-30,in=220] (5*360/6: 6.5cm);
\end{tikzpicture}\caption{A drawing of $G^6_{120} \Box P_n$ with $3n-3$ crossings. Each circle of vertices is one copy of $G^6_{120}$.\label{fig-120Pn}}\end{center}\end{figure}

\begin{theorem}Consider the cycle graph $C_n$. Then:
\begin{enumerate}\item $cr(G^6_{63} \Box C_n) = 2n$, for $n \geq 4$
\item $cr(G^6_{64} \Box C_n) = 2n$, for $n \geq 6$
\item $cr(G^6_{66} \Box C_n) = cr(G^6_{70} \Box C_n) = cr(G^6_{98} \Box C_n) = 3n$, for $n \geq 5$
\item $cr(G^6_{75} \Box C_n) = 2n$, for $n \geq 4$
\item $cr(G^6_{77} \Box C_n) = 2n$, for $n \geq 6$.
\item $cr(G^6_{92} \Box C_n) = 3n$, for $n \geq 4$\end{enumerate}\label{thm-cycles}\end{theorem}

\begin{proof}Consider graphs $G^6_{j}$ for $j \in \{41, 42, 47, 49, 53, 67\}$, all of which are displayed in Table \ref{tab-6vcart_cycle}, along with their crossing numbers, each of which were determined by Dra\v{z}ensk\'{a} and Kle\v{s}\v{c} \cite{drazenskaklesc2011}.


If we use $\subset$ to denote subgraphs, then the following can be easily verified. First, $G^6_{41} \subset G^6_{66} \subset G^6_{98}$, and $G^6_{41} \subset G^6_{70} \subset G^6_{98}$. Second, $G^6_{42} \subset G^6_{63}$. Third, $G^6_{47} \subset G^6_{64}$. Fourth, $G^6_{49} \subset G^6_{75}$. Fifth, $G^6_{53} \subset G^6_{77}$. Finally, $G^6_{67} \subset G^6_{92}$. It can be checked that these imply lower bounds for $cr(G^6_j \Box C_n)$ that meet the proposed values for each of $j \in \{63,$ $64,$ $66,$ $70,$ $75,$ $77,$ $92,$ $98$\}. Then, all that remains is to provide upper bounds that also meet the proposed values. Drawing methods which realise the proposed values for $j \in \{63, 64, 75, 77, 92, 98\}$ are displayed in Figures \ref{fig-63Cn}--\ref{fig-98Cn}. Note that the drawing method for $G^6_{98}$ also provides the corresponding drawing method for its subgraphs $G^6_{66}$ and $G^6_{70}$ if one removes the corresponding edges.\end{proof}

\begin{figure}[H]\begin{center}\begin{tikzpicture}[thick,scale=0.45]
\foreach \x in {0,5,10} {
	\foreach \y in {0,1,2,...,5} {
		\node[draw] at (\x,\y) {} ;
	};
	\draw (\x,5) -- (\x,0);
	\draw (\x,3) to[out=-10,in=90] (\x+0.75,2);
	\draw (\x+0.75,2) to[out=270,in=10] (\x,1);
	\draw (\x,2) to[out=190,in=90] (\x-0.75,1);
	\draw (\x-0.75,1) to[out=270,in=170] (\x,0);
};
\foreach \x in {0,1,2,...,5} {
	\draw[densely dotted] (-1,\x) -- (11,\x);
};
\end{tikzpicture}\caption{A drawing of $G^6_{63} \Box C_n$ with $2n$ crossings. The solid edges are the copies of $G^6_{63}$, while the dotted edges are those introduced by the Cartesian product.\label{fig-63Cn}}\end{center}\end{figure}

\begin{figure}[H]\begin{center}\begin{tikzpicture}[thick,scale=0.45]
\foreach \x in {0,5,10} {
	\foreach \y in {0,1,2,...,5} {
		\node[draw] at (\x,\y) {} ;
	};
	\draw (\x,5) -- (\x,0);
	\draw (\x,4) to[out=190,in=90] (\x-0.75,3);
	\draw (\x-0.75,3) to[out=270,in=170] (\x,2);
	\draw (\x,2) to[out=190,in=90] (\x-0.75,1);
	\draw (\x-0.75,1) to[out=270,in=170] (\x,0);
};
\foreach \x in {0,1,2,...,5} {
	\draw[densely dotted] (-1,\x) -- (11,\x);
};
\end{tikzpicture}\caption{A drawing of $G^6_{64} \Box C_n$ with $2n$ crossings. The solid edges are the copies of $G^6_{64}$, while the dotted edges are those introduced by the Cartesian product.\label{fig-64Cn}}\end{center}\end{figure}

\begin{figure}[H]\begin{center}\begin{tikzpicture}[thick,scale=0.45]
\foreach \x in {0,5,10} {
	\foreach \y in {0,1,2,...,5} {
		\node[draw] at (\x,\y) {} ;
	};
	\draw (\x,5) -- (\x,0);
	\draw (\x,4) to[out=-10,in=90] (\x+0.75,3);
	\draw (\x+0.75,3) to[out=270,in=10] (\x,2);
	\draw (\x,3) to[out=190,in=90] (\x-0.75,2);
	\draw (\x-0.75,2) to[out=270,in=170] (\x,1);
};
\foreach \x in {0,1,2,...,5} {
	\draw[densely dotted] (-1,\x) -- (11,\x);
};
\end{tikzpicture}\caption{A drawing of $G^6_{75} \Box C_n$ with $2n$ crossings. The solid edges are the copies of $G^6_{75}$, while the dotted edges are those introduced by the Cartesian product.\label{fig-75Cn}}\end{center}\end{figure}

\begin{figure}[H]\begin{center}\begin{tikzpicture}[thick,scale=0.45]
\foreach \x in {0,5,10} {
	\foreach \y in {0,1,2,...,5} {
		\node[draw] at (\x,\y) {} ;
	};
	\draw (\x,5) -- (\x,0);
	\draw (\x,5) to[out=-10,in=90] (\x+0.75,4);
	\draw (\x+0.75,4) to[out=270,in=10] (\x,3);
	\draw (\x,2) to[out=-10,in=90] (\x+0.75,1);
	\draw (\x+0.75,1) to[out=270,in=10] (\x,0);
};
\foreach \x in {0,1,2,...,5} {
	\draw[densely dotted] (-1,\x) -- (11,\x);
};
\end{tikzpicture}\caption{A drawing of $G^6_{77} \Box C_n$ with $2n$ crossings. The solid edges are the copies of $G^6_{77}$, while the dotted edges are those introduced by the Cartesian product.\label{fig-77Cn}}\end{center}\end{figure}

\begin{figure}[H]\begin{center}\begin{tikzpicture}[thick,scale=0.45]
\foreach \x in {0,5,10} {
	\foreach \y in {0,1,2,...,5} {
		\node[draw] at (\x,\y) {} ;
	};
	\draw (\x,5) -- (\x,0);
	\draw (\x,5) to[out=-10,in=90] (\x+0.75,4);
	\draw (\x+0.75,4) to[out=270,in=10] (\x,3);
	\draw (\x,3) to[out=-10,in=90] (\x+0.75,2);
	\draw (\x+0.75,2) to[out=270,in=10] (\x,1);
	\draw (\x,2) to[out=190,in=90] (\x-0.75,1);
	\draw (\x-0.75,1) to[out=270,in=170] (\x,0);
};
\foreach \x in {0,1,2,...,5} {
	\draw[densely dotted] (-1,\x) -- (11,\x);
};
\end{tikzpicture}\caption{A drawing of $G^6_{92} \Box C_n$ with $3n$ crossings. The solid edges are the copies of $G^6_{92}$, while the dotted edges are those introduced by the Cartesian product.\label{fig-92Cn}}\end{center}\end{figure}

\begin{figure}[H]\begin{center}\begin{tikzpicture}[thick,scale=0.45]
\foreach \x in {0,5,10} {
	\foreach \y in {0,1,2,...,5} {
		\node[draw] at (\x,\y) {} ;
	};
	\draw (\x,5) -- (\x,0);
	\draw (\x,4) to[out=-10,in=90] (\x+0.75,3);
	\draw (\x+0.75,3) to[out=270,in=10] (\x,2);
	\draw (\x,2) to[out=-10,in=90] (\x+0.75,1);
	\draw (\x+0.75,1) to[out=270,in=10] (\x,0);
	\draw (\x,3) to[out=190,in=90] (\x-0.75,2);
	\draw (\x-0.75,2) to[out=270,in=170] (\x,1);
};
\foreach \x in {0,1,2,...,5} {
	\draw[densely dotted] (-1,\x) -- (11,\x);
};
\end{tikzpicture}\caption{A drawing of $G^6_{98} \Box C_n$ with $3n$ crossings. The solid edges are the copies of $G^6_{98}$, while the dotted edges are those introduced by the Cartesian product.\label{fig-98Cn}}\end{center}\end{figure}

Each of the results in Theorems \ref{thm-59608390Cn} and \ref{thm-cycles} is stated for the Cartesian product of a graph and a sufficiently large cycle. However, for small cycles, the results are not provided in those Theorems. We present them now in Table \ref{tab-small}.

\begin{table}[H]\begin{center}$\begin{array}{|c|c|c|c|c|c|c|c|c|c|c|c|c|}
\hline {\bf i}                  & {\bf 59} & {\bf 60} & {\bf 63} & {\bf 64} & {\bf 66} & {\bf 70} & {\bf 75} & {\bf 77} & {\bf 83} & {\bf 90} & {\bf 92} & {\bf 98}\\
\hline {\bf cr(G^6_i \Box C_3)} & 8  & 8  & 6  & 6  & 7  & 7  & 6  & 6  & 10 & 11 & 9  & 9 \\
\hline {\bf cr(G^6_i \Box C_4)} & 16 & 16 &    & 8  & 12 & 12 &    & 8  & 16 & 16 &    & 12\\
\hline {\bf cr(G^6_i \Box C_5)} & 20 & 20 &    & 10 &    &    &    & 10 & 20 & 20 &    &   \\
\hline\end{array}$\caption{The crossing numbers for the Cartesian products of some six-vertex graphs with small cycles. Only those cases not already handled in Theorems \ref{thm-59608390Cn} and \ref{thm-cycles} are displayed.\label{tab-small}}\end{center}\end{table}

\begin{remark}To verify that the numbers provided in Table \ref{tab-small} are correct, we have submitted each graph to Crossing Number Web Compute \cite{chimaniwiedera2016,chimaniwiederasite2016}, an exact solver designed to handle sparse instances of small to moderate size. The proof files are available upon request from the corresponding author.\label{rem-small}\end{remark}

\begin{theorem}Consider the star graph $S_n$ for $n \geq 1$. Then, $cr(G^6_{62} \Box S_n) = 5\fl[n]\fl[n-1] + 2\fl[n]$.\label{thm-stars}\end{theorem}

\begin{proof}Consider the graph $G^6_{27}$, which is displayed in Table \ref{tab-6vcart_star}. The crossing number $cr(G^6_{27} \Box S_n) = 5\fl[n]\fl[n-1] + 2\fl[n]$ for $n \geq 1$ was determined by Kle\v{s}\v{c} and Schr\"{o}tter \cite{klescschrotter2013}. It is clear that $G^6_{27}$ is a subgraph of $G^6_{62}$. Hence, the lower bound is established. It can be verified that the drawing method for $G^6_{62} \Box S_n$, displayed in Figure \ref{fig-62Sn}, suffices to establish the upper bound.\end{proof}

\begin{figure}[H]\begin{center}\begin{tikzpicture}[thick,scale=0.6]
\foreach \x in {0,4,6,8,10,12,16} {
\foreach \y in {0,2,4,...,10} {
\node[draw] at (\x,\y) {} ;
};
\draw[densely dotted] (\x,10) -- (\x,0);
};
\foreach \x in {0,2,4,...,10} {
\draw (6,\x) -- (10,\x);
};
\foreach \x in {0,4,6} {
\draw[densely dotted] (\x,6) to[out=225,in=90] (\x-0.75,3);
\draw[densely dotted] (\x-0.75,3) to[out=270,in=135] (\x,0);
\draw[densely dotted] (\x,6) to[out=235,in=90] (\x-0.5,4);
\draw[densely dotted] (\x-0.5,4) to[out=270,in=125] (\x,2);
}
\foreach \x in {10,12,16} {
\draw[densely dotted] (\x,6) to[out=-45,in=90] (\x+0.75,3);
\draw[densely dotted] (\x+0.75,3) to[out=270,in=45] (\x,0);
\draw[densely dotted] (\x,6) to[out=-55,in=90] (\x+0.5,4);
\draw[densely dotted] (\x+0.5,4) to[out=270,in=55] (\x,2);
}
\draw[densely dotted] (8,6) to[out=235,in=90] (8-0.5,4);
\draw[densely dotted] (8-0.5,4) to[out=270,in=125] (8,2);
\foreach \x in {10,8} {
\draw (8,\x) to[out=160,in=0] (4,\x+0.5);
\draw (4,\x+0.5) to[out=180,in=20] (0,\x);
\draw (8,\x) to[out=20,in=180] (12,\x+0.5);
\draw (12,\x+0.5) to[out=0,in=160] (16,\x);
\draw (8,\x) to[out=165,in=0] (6,\x+0.3);
\draw (6,\x+0.3) to[out=180,in=15] (4,\x);
\draw (8,\x) to[out=15,in=180] (10,\x+0.3);
\draw (10,\x+0.3) to[out=0,in=165] (12,\x);
}
\draw (8,6) to[out=135,in=0] (4,7.25);
\draw (4,7.25) to[out=180,in=45] (0,6);
\draw (8,6) to[out=45,in=180] (12,7.25);
\draw (12,7.25) to[out=0,in=135] (16,6);
\draw (8,6) to[out=145,in=0] (6,6.75);
\draw (6,6.75) to[out=180,in=35] (4,6);
\draw (8,6) to[out=35,in=180] (10,6.75);
\draw (10,6.75) to[out=0,in=145] (12,6);
\draw (8,4) to[out=110,in=0] (6,6.25);
\draw (6,6.25) to[out=180,in=70] (4,4);
\draw (8,4) to[out=105,in=0] (6,6.5);
\draw (6,6.5) to[out=180,in=0] (3,6.5);
\draw (3,6.5) to[out=180,in=70] (0,4);
\draw (8,4) to[out=75,in=180] (10,6.5);
\draw (10,6.5) to[out=0,in=180] (13,6.5);
\draw (13,6.5) to[out=0,in=110] (16,4);
\draw (8,2) to[out=255,in=0] (6,-0.5);
\draw (6,-0.5) -- (3,-0.5);
\draw (3,-0.5) to[out=180,in=300] (0,2);
\draw (8,2) to[out=285,in=180] (10,-0.5);
\draw (10,-0.5) -- (13,-0.5);
\draw (13,-0.5) to[out=0,in=240] (16,2);
\draw (8,4) to[out=70,in=180] (10,6.25);
\draw (10,6.25) to[out=0,in=110] (12,4);
\draw (8,0) to[out=-45,in=180] (10,-1);
\draw (10,-1) -- (13,-1);
\draw (13,-1) to[out=0,in=225] (16,0);
\draw (8,0) to[out=225,in=0] (6,-1);
\draw (6,-1) -- (3,-1);
\draw (3,-1) to[out=180,in=-45] (0,0);
\draw (8,0) to[out=-35,in=180] (10,-0.75);
\draw (10,-0.75) to[out=0,in=215] (12,0);
\draw (8,0) to[out=215,in=0] (6,-0.75);
\draw (6,-0.75) to[out=180,in=-35] (4,0);
\draw (8,2) to[out=250,in=0] (6,-0.25);
\draw (6,-0.25) to[out=180,in=285] (4,2);
\draw (8,2) to[out=-70,in=180] (10,-0.25);
\draw (10,-0.25) to[out=0,in=250] (12,2);
\draw[densely dotted] (8,0) to[out=235,in=0] (6,-1.25);
\draw[densely dotted] (6,-1.25) to[out=180,in=0] (1,-1.25);
\draw[densely dotted] (1,-1.25) to[out=180,in=270] (-1,2);
\draw[densely dotted] (-1,2) to[out=90,in=270] (-1,5);
\draw[densely dotted] (-1,5) to[out=90,in=180] (2,7.5);
\draw[densely dotted] (2,7.5) to[out=0,in=180] (4,7.5);
\draw[densely dotted] (4,7.5) to[out=0,in=125] (8,6);
\begin{scope}
\node[draw=none,fill=none,transform shape] () at (2,2.5) {\Huge{$\dots$}};
\node[draw=none,fill=none,transform shape] () at (14,2.5) {\Huge{$\dots$}};
\end{scope}
\draw [decorate,decoration={brace,amplitude=5pt,mirror,raise=4ex}]
(0,-0.75) -- (6.5,-0.75) node[draw=none,fill=none,transform shape,midway,yshift=-5.75em]{\Large{$\lfloor \frac{n}{2}\rfloor$}};
\draw [decorate,decoration={brace,amplitude=5pt,mirror,raise=4ex}]
(9.5,-0.75) -- (16,-0.75) node[draw=none,fill=none,transform shape,midway,yshift=-5.75em]{\Large{$\lceil \frac{n}{2}\rceil$}};
\end{tikzpicture}\caption{A drawing of $G^6_{62} \Box S_n$ with $5\fl[n]\fl[n-1] + 2\fl[n]$ crossings. The dotted edges are the copies of $G^6_{62}$, while the solid edges are those introduced by the Cartesian product. There are $\fl[n]$ copies of $G^6_{62}$ on the left, and $\lceil\frac{n}{2}\rceil$ copies of $G^6_{62}$ on the right.\label{fig-62Sn}}\end{center}\end{figure}

\end{document}